\documentclass[11pt]{amsart}

\usepackage{latexsym,rawfonts}
\usepackage{amsfonts,amssymb}
\usepackage{amsmath,amsthm}

\usepackage[plainpages=false]{hyperref}
\usepackage{graphicx}
\usepackage{color}
\usepackage[table]{xcolor}
\usepackage{longtable}

\newcommand{\R}{\mathbb{R}}

\newcommand{\rnnn}{\mathbb R^n}

\newcommand{\sn}{ {\mathbb{S}^{n-1}}}

\newcommand{\psum}{{+_{\negthinspace\kern-2pt p}}\,}
\newcommand{\qsum}[1]{{+_{\negthinspace\kern-2pt #1}}\,}
\newcommand{\dpsum}{{\tilde+_{\negthinspace\kern-1pt p}}\,}
\newcommand{\dqsum}[1]{{\tilde+_{\negthinspace\kern-1pt #1}}\,}
\newcommand{\lsub}[1]{\hskip -1.5pt\lower.5ex\hbox{$_{#1}$}}

\numberwithin{equation}{section}

\newtheorem{theo}{Theorem}[section]

\newtheorem{lem}[theo]{Lemma}
\newtheorem{prop}[theo]{Proposition}

\newtheorem{rem}[theo]{Remark} \theoremstyle{definition}

\begin{document}

\title{From the Brunn-Minkowski inequality to a class of  \\
generalized Poincar\'{e}-type inequalities for torsional rigidity
}

\author[N. Fang]{Niufa Fang}
\address{School of Mathematics, Hunan University, Changsha, 410082, Hunan Province, China}
\email{fangniufa@hnu.edu.cn}

\author[J. Hu]{Jinrong Hu}
\address{School of Mathematics, Hunan University, Changsha, 410082, Hunan Province, China}
\email{hujinrong@hnu.edu.cn}

\author[L. Zhao]{Leina Zhao}
\address{College of Mathematics and Statistics, Chongqing Jiaotong University, Chongqing, 400074, China}
\email{zhao\_leina@163.com}

\begin{abstract}
In this paper, we establish a class of generalized Poincar\'{e}-type inequalities for torsional rigidity on the boundary of a convex body of class $C^{2}_{+}$ in $\rnnn$ by using the concavity of related Brunn-Minkowski inequality.

\end{abstract}
\thanks{The research of the  first author  was supported by  NSFC (No. 12001291) and  the Natural Science Foundation of Hunan Province (No. 2022JJ30117)}
\maketitle

\baselineskip18pt

\parskip3pt

\section{Introduction}

The classical Brunn-Minkowski inequality plays an important role in the theory of convex bodies, which states that if $\Omega$ and $\widetilde{\Omega}$ are convex bodies in $\rnnn$ and $t\in [0,1]$, then
\begin{equation}\label{VE}
V((1-t)\Omega+t\widetilde{\Omega})^{1/n}\geq (1-t)V(\Omega)^{1/n}+tV(\widetilde{\Omega})^{1/n},
\end{equation}
where $V$ is the $n$-dimensional volume. For standard references  to the proof of \eqref{VE}, books such as Gardner \cite{G06} and Schneider \cite{S14} are recommended. In the beautiful survey, Gardner \cite{Ga02} illustrated that \eqref{VE} has many applications in \emph{Geometry and Analysis}. For instance,  the classical isoperimetric inequality can be deduced by using \eqref{VE}. Besides, by means of \eqref{VE}, Colesanti \cite{CA08} gave a different argument to derive Poincar\'{e} type inequalities on the boundary of convex bodies, while Poincar\'{e} and Brunn-Minkowski inequalities on the boundary of weighted Riemannian mainfolds were obtained by Kolesnikov-Milman \cite{KM18}. With the development of the classical Brunn-Minkowski inequality, several Brunn-Minkowski type inequalities have also been obtained for well-known functionals in the \emph{Calculus of Variations}, such as the first eigenvalue of the Laplacian, the Newton capacity and the torsional rigidity, see \cite{CA05, HSX18}. However, to our best of knowledge, researches on the Poincar\'{e} type inequalities for above functionals associated with boundary value problems are very few.

In this paper, we focus on torsional rigidity and add a new contribution on Poincar\'{e} type inequality for torsional rigidity. Let $\Omega$ be a convex body in the $n$-dimensional Euclidean space ${\rnnn}$. The torsional rigidity of $\Omega$ is defined as
\begin{equation}\label{cdf1*}
\frac{1}{T(\Omega)}=\inf\left\{\frac{\int_{\Omega}|\nabla U|^{2}{d}X}{(\int_{ \Omega}|U|{d}X)^{2}}:\ U\in W^{1,2}_{0}(\Omega)\ , \int_{\Omega}|U|{d}X> 0\right\}.
\end{equation}
For problem \eqref{cdf1*}, from \cite[p. 112]{CA05}, we know that there exists a unique function $U$ such that
\begin{equation}\label{cdf2*}
T(\Omega)=\int_{\Omega} |\nabla U|^{2}dX,
\end{equation}
where  $U$ is the solution of the boundary-value problem
\begin{equation}\label{capequ*}
\left\{
\begin{array}{lr}
\Delta U(X)= -2, & X\in \Omega, \\
U(X)=0,  & X\in  \partial \Omega.
\end{array}\right.
\end{equation}

 For two arbitrary convex bodies $\Omega, \widetilde{\Omega}$ in ${\rnnn}$, in analogy with the variational formula of volume (see, e.g., ~\cite{S14}),  Colesanti-Fimiani~\cite{CF10} obtained the variational formula of torsional rigidity,
\begin{equation}\label{capha*}
\frac{d}{dt}T(\Omega+t\widetilde{\Omega})\Big|_{t=0^+}=\int_{\partial \Omega}h(\widetilde{\Omega},\nu_{\Omega}(X))|\nabla U(X)|^{2}{d}\mathcal{H}^{n-1}(X),
\end{equation}
 where $h(\widetilde{\Omega},\cdot)$ is the support function of $\widetilde{\Omega}$, $\mathcal{H}^{n-1}(\cdot)$ is the $(n-1)$-dimensional Hausdorff measure, and $\nu_{\Omega}$ is the Gauss map of $\partial \Omega$.

In view of \eqref{capha*}, the torsional measure $\mu_{tor}(\Omega,\eta)$ is defined on the unit sphere ${\sn}$ by
\begin{equation}\label{capme*}
\mu_{tor}(\Omega,\eta)=\int_{F(\eta)}|\nabla U(X)|^{2}{d}\mathcal{H}^{n-1}(X)=\int_{\eta}|\nabla U(F(\xi))|^{2}dS_{\Omega}(\xi)
\end{equation}
for every Borel subset $\eta$ of ${\sn}$. Here $S_{\Omega}$ is the surface area measure of $\Omega$, and $F$ is the inverse of the Gauss map. Moreover, the
 estimates of harmonic function proved by Dahlberg \cite{D77} illustrate that $\nabla U$ has finite non-tangential limit at $\mathcal{H}^{n-1}$-a.e. on $\partial \Omega$ with $|\nabla U| \in L^{2}(\partial \Omega, \mathcal{H}^{n-1})$, which is not limited by the smooth condition on $\Omega$. Hence, $\eqref{capme*}$ is well defined a.e. on the unit sphere ${\sn}$.

The Brunn-Minkowski inequality for torsional rigidity tells that if $\Omega$ and $\widetilde{\Omega}$ are convex bodies in $\rnnn$ and $t\in [0,1]$, then
\begin{equation}\label{TE}
T((1-t)\Omega+t\widetilde{\Omega})^{1/(n+2)}\geq (1-t)T(\Omega)^{1/(n+2)}+tT(\widetilde{\Omega})^{1/(n+2)}.
\end{equation}
The above inequality was first gained by Borell \cite{BC85}, and Colesanti \cite{CA05} proved that equality occurs in \eqref{TE} if and only if $\Omega$ is homothetic to $\widetilde{\Omega}$.

 Denoted by ${\rm II}_{\partial \Omega}$ the second fundamental form of $\partial \Omega$. Inspired by \cite{CA08,J96}, we put forward an argument which illustrates that how to transfer the Brunn-Minkowski inequality to a Poincar\'{e}-type inequality for torsional rigidity on smooth convex body. The main result is showed as follows.

\begin{theo}\label{main}
Let $\Omega\subset \rnnn$ be a convex body of class $C^{2}_{+}$. For every $\psi \in  C^{1}(\partial \Omega)$, if
\begin{equation}\label{PE}
\int_{\partial \Omega}\psi|\nabla U|^{2}d \mathcal{H}^{n-1}(X)=0,
\end{equation}
then
\begin{equation}
\begin{split}
\label{tor2}
&-\int_{\partial \Omega}{\rm tr}({\rm II}_{\partial \Omega})\psi^{2}|\nabla U|^{2}d \mathcal{H}^{n-1}(X)-2 \int_{\partial \Omega }\psi(\nabla \dot{U}(X)\cdot \nu_{\Omega})|\nabla U|d \mathcal{H}^{n-1}(X)\\
&\quad \quad +4\int_{\partial \Omega}\psi^{2}|\nabla U|d \mathcal{H}^{n-1}(X)\leq \int_{\partial \Omega}({\rm II}_{\partial \Omega}^{-1}\nabla_{\partial \Omega}\psi\cdot \nabla_{\partial \Omega} \psi)|\nabla U|^{2}d \mathcal{H}^{n-1}(X),
\end{split}
\end{equation}
where  $\dot{U}$ satisfies the following equation
\begin{equation*}
\label{Ut}
\left\{
\begin{array}{lr}
\Delta \dot{U}=0, & X\in \Omega. \\
\dot{U}(X)=|\nabla U(X)|\psi.  & X\in  \partial \Omega.
\end{array}\right.
\end{equation*}

\end{theo}

 It is worth mentioning that Colesanti  \cite{CA08} has used the classical Brunn-Minkowski inequality for volume to derive a class of Poincar\'{e} inequalities for surface area measure. Theorem \ref{main} looks similar to Colesanti's inequality,  but the former requires a more complicated proof for which highly depends on the property of solution of Dirichlet problem \eqref{capequ*}. Theorem \ref{main} may be a new contribution to potential theory, analysis, fully nonlinear equations and geometric inequality. To a certain extent, it shows us that how these fields can be used together.

\section{Preliminaries}
\label{Sec2}
In this section, we list some basics on convex bodies and torsional rigidity.
\subsection{Basics of convex bodies}
There are many standard references with regard to convex bodies, for instance, please refer to Gardner \cite{G06} and Schneider \cite{S14}.

 Denote by ${\rnnn}$ the $n$-dimensional Euclidean space. For $Y, Z\in {\rnnn}$, $ Y\cdot Z $ denotes the standard inner product. For $X\in{\rnnn}$, $|X|=\sqrt{ X\cdot X}$ is the Euclidean norm. Let ${\sn}$ be the unit sphere, and $C({\sn})$ be the set of continuous functions defined on the unit sphere ${\sn}$. A compact convex set of ${\rnnn}$ with non-empty interior is called as a convex body.

For a compact convex set $\Omega$ in ${\rnnn}$, the diameter of $\Omega$ is denoted by
\[
{\rm diam}(\Omega)=\max\{|X-Y|:X,Y \in \Omega\}.
\]

If $\Omega$ is a convex body containing the origin in ${\rnnn}$, for $\xi\in{\sn}$, the support function of $\Omega$ (with respect to the origin) is defined by
\[
h(\Omega,\xi)=\max\{ \xi\cdot Y:Y \in \Omega\}.
\]
Extending this definition to a homogeneous function of degree one in $\rnnn\backslash \{0\}$ by the equation $h(\Omega, \xi)=|\xi|h\left(\Omega, \frac{\xi}{|\xi|}\right)$.

 The map $\nu_{\Omega}:\partial \Omega\rightarrow {\sn}$ denotes the Gauss map of $\partial\Omega$. Meanwhile, for $\omega\subset {\sn}$, the inverse of Gauss map $\nu_{\Omega}$ is expressed as
\begin{equation*}
\nu^{-1}_{\Omega}(\omega)=\{X\in \partial \Omega:  \nu_{\Omega}(X) {\rm \ is \ defined \ and }\ \nu_{\Omega}(X)\in \omega\}.
\end{equation*}
For simplicity in the subsequence, we abbreviate $\nu^{-1}_{\Omega}$ as $F$. In particular, for a convex body $\Omega$ being of class $C^{2}_{+}$, i.e., its boundary is of class $C^{2}$ and of positive Gauss curvature, the support function of $\Omega$ can be written as
\begin{equation}\label{hhom}
h(\Omega,\xi)=\xi\cdot F(\xi)=\nu_{\Omega}(X)\cdot X, \ {\rm where} \ \xi\in {\sn}, \ \nu_{\Omega}(X)=\xi \ {\rm and} \ X\in \partial \Omega.
\end{equation}
 Let $\{e_{1},e_{2},\ldots, e_{n-1}\}$ be a local orthonormal frame on ${\sn}$, $h_{i}$ be the first order covariant derivatives of $h(\Omega,\cdot)$  with respect to a local orthonormal frame on ${\sn}$. Differentiating \eqref{hhom} with respect to $e_{i}$ , we get
\[
h_{i}=e_{i}\cdot F(\xi)+\xi\cdot F_{i}(\xi).
\]
Since $F_{i}$ is tangent to $ \partial \Omega$ at $F(x)$, we obtain
\begin{equation}\label{Fi}
h_{i}=e_{i}\cdot F(\xi).
\end{equation}
Combining \eqref{hhom} and \eqref{Fi}, we have (see also \cite[p. 97]{U91})
\begin{equation}\label{Fdef}
F(\xi)=\sum_{i} h_{i}(\Omega,\xi)e_{i}+h(\Omega,\xi)\xi=\nabla_{\sn}h(\Omega,\xi)+h(\Omega,\xi)\xi.
\end{equation}
Here $\nabla_{\sn}$ is the spherical gradient. On the other hand, since we can extend $h(\Omega,\cdot)$ to $\rnnn$ as a 1-homogeneous function $h(\Omega, \cdot)$, then restrict the gradient of $h(\Omega,\cdot)$ on $\sn$, it yields that (see for example, \cite{CY76})
\begin{equation}\label{hf}
\nabla h(\Omega,\xi)=F(\xi), \ \forall \xi\in{\sn},
\end{equation}
where $\nabla$ is the gradient operator in $\rnnn$. Let $h_{ij}$ be the second order covariant derivatives of $h(\Omega,\cdot)$ with respect to a local orthonormal frame on ${\sn}$. Then, applying \eqref{Fdef} and \eqref{hf}, we have (see, e.g., \cite[p. 382]{J91})
\begin{equation}\label{hgra}
\nabla h(\Omega,\xi)=\sum_{i}h_{i}e_{i}+h\xi, \quad F_{i}(\xi)=\sum_{j}(h_{ij}+h\delta_{ij})e_{j}.
\end{equation}
Also, the {\rm reverse Weingarten map} of $\Omega$ at $\xi$ is given by
\[
(h_{ij}+h\delta_{ij})=D(\nu^{-1}_{\Omega}) \quad {\rm on} \ \sn.
\]

The second fundamental form ${\rm II}_{\partial \Omega}$ at $X\in \partial \Omega$ is defined as ${\rm II}_{X}(Y,Z)=\nabla_{Y}\nu\cdot Z$ with $Y,Z\in T_{X}\partial \Omega$.

From the definition of Weingarten map, one see that the second fundamental form
of $\partial \Omega$ is equivalent to Weingarten map, i.e.,
\[
{\rm II}_{\partial \Omega}=D \nu_{\Omega}.
\]

The Gauss curvature of $\partial\Omega$, $\kappa$, i.e., the Jacobian determinant of Gauss map of $\partial\Omega$, is revealed as
\begin{equation*}
\kappa=\frac{1}{\det(h_{ij}+h\delta_{ij})}.
\end{equation*}

\subsection{Basics of torsional rigidity} For reader's convenience, we collect some facts about the properties of torsional rigidity.

\subsubsection{The  properties of torsional rigidity}

 Now, according to \cite{CF10, D77}, we first list some properties corresponding to the solution $U$ of \eqref{capequ*}, as follows.

For any $ X\in\partial \Omega$, $0<b<1$, the non-tangential cone is defined as
\begin{equation*}\label{capm2}
\Gamma(X)=\left\{Y\in \Omega:dist(Y,\partial \Omega)> b|X-Y|\right\}.
\end{equation*}

\begin{lem}\label{Nonfin}
Let $\Omega$ be a convex body in ${\rnnn}$ and $U$ be the solution of \eqref{capequ*} in $\Omega$. Then the non-tangential limit
\begin{equation*}\label{capm3}
\nabla U(X)=\lim_{Y\rightarrow X, \ Y \in \Gamma(X)}\nabla U(Y)
\end{equation*}
exists for $\mathcal{H}^{n-1}$ almost all $X\in\partial \Omega$. Furthermore, for $\mathcal{H}^{n-1}$ almost all $X\in \partial \Omega$,
\begin{equation*}
\nabla U(X)=-|\nabla U(X)|\nu_{\Omega}(X)\quad{\rm and}\quad |\nabla U|\in L^{2}(\partial \Omega,  \mathcal{H}^{n-1}).
\end{equation*}
\end{lem}

\begin{lem}\label{argu}Let $\Omega$ be a convex body in ${\rnnn}$ and $U$ be the solution of \eqref{capequ*} in $\Omega$. Then
\begin{equation*}\label{DM}
|\nabla U(X)|\leq {\rm diam}(\Omega),\ \forall X\in  \Omega.
\end{equation*}
\end{lem}

\begin{lem}[\cite{BC85}]\label{TBM} If $\Omega$ and $\widetilde{\Omega}$ are convex bodies in $\rnnn$ and $t\in [0,1]$, then
\begin{equation}\label{ji}
T((1-t)\Omega+t\widetilde{\Omega})^{1/(n+2)}\geq (1-t)T(\Omega)^{1/(n+2)}+tT(\widetilde{\Omega})^{(1/(n+2)}.
\end{equation}
Colesanti \cite{CA05} proved that equality occurs in \eqref{ji} if and only if $\Omega$ is homothetic to $\widetilde{\Omega}$.

\end{lem}
\begin{rem}\label{caphom}[The homogeneity of torsional rigidity]
If $U$ is the solution of \eqref{capequ*} in $\Omega$ and $a> 0$, then the function
\begin{equation*}\label{}
V(Y)=a^{2}U\left(\frac{Y}{a}\right), \quad Y\in a\Omega,
\end{equation*}
is the solution in $a\Omega=\left\{Y=aZ\big| \ Z\in \Omega\right\}$. It follows that the torsional rigidity is positively homogeneous of order $n+2$, i.e., $T(a\Omega)=a^{n+2}T(\Omega)$.

\end{rem}

 Let $\Omega$ be a convex body of class $C^{2}_{+}$, and $U$ be the solution of \eqref{capequ*} in $\Omega$. By a direct calculation in conjunction with Lemma \ref{Nonfin} and the divergence theorem, \cite{CA05} showed that
\begin{equation}\label{capdef2}
T(\Omega)=\frac{1}{n+2}\int_{\partial \Omega}|\nabla U(X)|^{2}(X\cdot \nu_{\Omega}(X)){d}{\mathcal{H}^{n-1}}(X).
\end{equation}
Since for every $g\in C({\sn})$,
\begin{equation*}\label{arsphere}
\int_{{\sn}}g(\xi){d}{S_{\Omega}(\xi)}=\int_{\partial \Omega}g(\nu_{\Omega}(X)){d}{\mathcal{H}^{n-1}}(X).
\end{equation*}
\eqref{capdef2} is equivalent to
\begin{align}\label{captatget2}
T(h)&=\frac{1}{n+2}\int_{{\sn}}|\nabla U(F(\xi))|^{2}h(\xi){d}{S_{\Omega}(\xi)}\notag\\
&=\frac{1}{n+2}\int_{{\sn}}|\nabla U(F(\xi))|^{2}h(\xi)\det(h_{ij}(\xi)+h(\xi)\delta_{ij}){d}{\xi}.
\end{align}
Let us define the set
\begin{equation*}\label{S}
\mathcal{S}=\{h\in C^{2}({\sn}):h_{ij}+h\delta_{ij}>0 \  {\rm on} \ \sn \}.
\end{equation*}
By using Lemma \ref{TBM} and \eqref{captatget2}, we get
\begin{lem}\label{P} $T^{1/(n+2)}$ is concave in $\mathcal{S}$, i.e.,
\begin{equation*}
T^{1/(n+2)}((1-t)h_{0}+th_{1})\geq (1-t)T^{1/(n+2)}(h_{0})+tT^{1/(n+2)}(h_{1}), \ \forall h_{0}, h_{1}\in \mathcal{S}, \ \forall t\in [0,1].
\end{equation*}
\end{lem}

\section{Proof of Theorem \ref{main} }
\label{Sec3}
In this section, our aim is to prove Theorem \ref{main}. We start with getting the following result.
\begin{theo}\label{main2}
Let $h\in \mathcal{S}$, $\phi\in C^{1}(\sn)$. If

\begin{equation}\label{T31}
\int_{\sn} \phi |\nabla U(F(\xi))|^{2}\det(h_{ij}+h\delta_{ij})d \xi=0,
\end{equation}
then
\begin{equation}
\begin{split}
\label{T32}
&-\int_{\sn}{\rm tr}(c_{ij})|\nabla U(F(\xi))|^{2}\phi^{2}d \xi-2\int_{\sn}\phi\frac{|\nabla U(F(\xi))|^{2}}{\kappa}\nabla \dot{U}(F(\xi))\cdot \xi d \xi \\
&\quad +4\int_{\sn}\phi^{2}\frac{|\nabla U(F(\xi))|}{\kappa}d\xi \leq \int_{\sn} \sum_{i,j}c_{ij}|\nabla U(F(\xi))|^{2}\phi_{i}\phi_{j}d \xi,
\end{split}
\end{equation}
where $c_{ij}$ is the cofactor matrix of $(h_{ij}+h\delta_{ij})$.
\end{theo}
Here we remark that Theorem \ref{main} and Theorem \ref{main2} can be obtained by each other through the change of variable given by the Gauss map of $\partial \Omega$.

Now, in order to prove Theorem \ref{main2}, we first list the following results proved similarly in \cite[Lemma 2.5]{J96}. For reader's convenience, we give the proof process.

\begin{lem}\label{UPO} Let $\Omega_{t}$ be a convex body of class $C^{2}_{+}$ in ${\rnnn}$, and $U(\cdot,t)$ be the solution of \eqref{capequ*} in $\Omega_{t}$, then

(i)$(\nabla^{2} U(F(\xi,t),t)e_{i})\cdot e_{j}=-\kappa |\nabla U(F(\xi,t),t)|c_{ij}(\xi,t)$;

(ii)$(\nabla^{2} U(F(\xi,t),t)\xi)\cdot \xi=\kappa |\nabla U(F(\xi,t),t)|\sum_{i} c_{ii}(\xi,t)-2$;

(iii)$(\nabla^{2} U(F(\xi,t),t)e_{i})\cdot \xi=-\kappa\sum_{j} |\nabla U(F(\xi,t),t)|_{j}c_{ij}(\xi,t)$.
\end{lem}

\begin{proof}
Assume that $h(\xi,t)\in C^{2}({\sn}\times (0,\infty))$ is the support function of $\Omega_{t}$ and let $\phi=\frac{\partial h(\xi,t)}{\partial t}$. Then, recall \eqref{hf} and \eqref{hgra}, we get $F(\xi,t)=\sum_{i}h_{i}e_{i}+h\xi$,  $\frac{\partial F(\xi,t)}{\partial t}:=\dot{F}(\xi,t)=\sum_{i}\phi_{i}e_{i}+\phi \xi$, $F_{i}(\xi,t)=\sum_{j}w_{ij}e_{j}$ with $w_{ij}:=h_{ij}+h\delta_{ij}$, and $F_{ij}(\xi,t)=\sum_{k}w_{ijk}e_{k}-w_{ij}\xi$, where $w_{ijk}$ are the covariant derivatives of $w_{ij}$. Note that, from Lemma \ref{Nonfin}, we know that $\nabla U(F(\xi,t),t)=-|\nabla U(F(\xi,t),t)|\xi$, thus we have
\begin{equation*}\label{U1Fdef}
 \nabla U \cdot F_{i}=0,
\end{equation*}
and
\begin{equation*}\label{U2Fdef}
 ((\nabla^{2}U)F_{j})\cdot F_{i}+\nabla U\cdot F_{ij}=0.
\end{equation*}
It follows that
\begin{equation}\label{U3Fdef}
\sum_{k,l} w_{ik}w_{jl}((\nabla^{2}U)e_{l})\cdot e_{k})+w_{ij}|\nabla U|=0.
\end{equation}
Multiplying both sides of ~\eqref{U3Fdef} by $c_{ip}c_{jq}$ and sum for $i, j$, we get
\begin{equation*}\label{U4Fdef}
 \sum_{k,l}\delta_{kp}\delta_{ql}((\nabla^{2}U e_{l})\cdot e_{k})=-\kappa \sum_{j}\delta_{jp}c_{jq}|\nabla U|.
\end{equation*}
Hence,
\begin{equation*}\label{U7Fdef}
(\nabla^{2}U e_{i})\cdot e_{j}=-\kappa c_{ij}|\nabla U|.
\end{equation*}
This proves $(i)$.

Second, recall that
\begin{equation*}\label{U8Fdef}
|\nabla U(F(\xi,t),t)|=-\nabla U(F(\xi,t),t)\cdot \xi.
\end{equation*}
Taking the covariant derivative of both sides, we have
\begin{align}\label{U9Fdef}
|\nabla U|_{j}&=-\nabla U\cdot e_{j}-(\nabla^{2}U)F_{j}\cdot \xi\notag\\
&=- \sum_{i} w_{ij}((\nabla ^{2}U)e_{i}\cdot \xi).
\end{align}
Multiplying both sides of \eqref{U9Fdef} by $c_{lj}$ and sum for $j$, and combining
\begin{equation*}\label{U10Fdef}
\sum_{j}c_{lj}w_{ij}=\delta_{li}\det(h_{ij}+h\delta_{ij}).
\end{equation*}
Then, we get
\begin{equation*}\label{U11Fdef}
\sum_{j}c_{ij}|\nabla U|_{j}=-\det(h_{ij}+h\delta_{ij})(\nabla^{2}U) e_{i}\cdot \xi.
\end{equation*}
Hence
\begin{equation*}\label{U12Fdef}
(\nabla^{2}U) e_{i}\cdot \xi=-\kappa \sum_{j}c_{ij}|\nabla U|_{j}.
\end{equation*}
This proves $(iii)$.

Last, employing $(i)$, $div_{\Gamma}(\nabla U)=div(\nabla U)-(\nabla^{2}U)\xi \cdot \xi$, i.e., the divergence is divided into tangential and normal components, and $\Delta U=-2$, thus we obtain
\begin{align*}\label{DIV}
(\nabla^{2}U)\xi\cdot \xi&=(\nabla^{2} U)\xi\cdot \xi-\Delta U-2\notag\\
&=-\sum_{i}(\nabla^{2}U)e_{i}\cdot e_{i}-2=\kappa \sum_{i}c_{ii}|\nabla U|-2.
\end{align*}
Hence, the proof is completed.
\end{proof}

\begin{prop}\label{fsd}
Let $h\in \mathcal{S}$, $\phi\in C^{\infty}(\sn)$ and $\varepsilon>0$ such that $h+t\phi \in \mathcal{S}$ for every $t\in (-\varepsilon,\varepsilon)$. Set $h_{t}=h+t\phi$ and $f(t)=T(h_{t})$. Then
\begin{equation*}\label{}
f^{'}(t)=\int_{\sn}\phi G(h_{t})d \xi, \quad \forall t\in (-\varepsilon,\varepsilon),
\end{equation*}
and
\begin{equation*}
\begin{split}
\label{Up3}
f^{''}(0)&=-2\int_{\sn}\phi\frac{|\nabla U(F(\xi))|}{\kappa}\nabla \dot{U}(F(\xi))\cdot \xi d \xi-\int_{\sn}\sum_{i}|\nabla U(F(\xi))|^{2}c_{ii}\phi^{2}d \xi\\
&\quad+\int_{\sn}\sum_{i,j}\phi(c_{ij}|\nabla U(F(\xi))|^{2}\phi_{i})_{j}d \xi+4\int_{\sn}\phi^{2}\frac{|\nabla U(F(\xi))|}{\kappa}d\xi,
\end{split}
\end{equation*}
where
\begin{equation*}\label{Fdef1}
G(h_{t})=|\nabla  U(F(\xi,t),t)|^{2}\det((h_{t})_{ij}+h_{t}\delta_{ij})(\xi)
\end{equation*}
with $h_{t}$ is the support function of $\Omega_{t}$ and $\dot{U}$ satisfies the following equation
\begin{equation*}
\label{Ut}
\left\{
\begin{array}{lr}
\Delta \dot{U}=0, & X\in \Omega. \\
\dot{U}(X)=|\nabla U(X)|\phi(\nu_{\Omega}(X)),  & X\in  \partial \Omega.
\end{array}\right.
\end{equation*}
\end{prop}
\begin{proof}
 By adopting similar arguments as \cite[Lemma 3.1]{C15}, we first conclude that $U(F(\xi,t),t)$ is differentiable with respect to $t$. Based on this fact,  then we directly compute
\begin{equation}
\begin{split}
\label{capcalcu}
f^{'}(t)&=\frac{d}{ds}T(h_{t}+s\phi)\big|_{s=0}\\
&=\frac{1}{n+2}\int_{{\sn}}(\phi G(h_{t})+h_{t}\frac{d}{dt}G(h_{t}))d \xi, \ \forall \xi\in {\sn},
\end{split}
\end{equation}
where
\begin{equation}\label{Fdef1*}
G(h_{t})=|\nabla  U(F(\xi,t),t)|^{2}\det((h_{t})_{ij}+h_{t}\delta_{ij})(\xi),
\end{equation}
 with
 \begin{equation*}\label{Fdef2}
 F(\xi)=\nu^{-1}_{\Omega}(\xi)=\nabla h(\xi)=\sum_{i}h_{i}(\xi)e_{i}(\xi)+h(\xi)\xi.
\end{equation*}
By \eqref{Fdef1*}, we have
 \begin{equation}
\begin{split}
\label{Gcal}
\frac{d}{dt}G(h_{t})&=\frac{d}{ds}G(h_{t}+s\phi)\big|_{s=0}=\frac{d}{dt}[|\nabla  U(F(\xi,t),t)|^{2}\det((h_{t})_{ij}+h_{t}\delta_{ij})]\\
&=|\nabla U(F(\xi,t),t)|^{2}(c^{t}_{ij}(\phi_{ij}+\phi\delta_{ij}))+\det((h_{t})_{ij}+h_{t}\delta_{ij})\frac{d}{dt}(|\nabla U(F(\xi,t),t)|^{2})\\
&=|\nabla U(F(\xi,t),t)|^{2}(c^{t}_{ij}(\phi_{ij}+\phi\delta_{ij}))+2|\nabla U(F(\xi,t),t)|\det((h_{t})_{ij}+h_{t}\delta_{ij})\frac{d}{dt}(|\nabla U(F(\xi,t),t)|),
\end{split}
\end{equation}
where $c^{t}_{ij}$ is the cofactor matrix of $((h_{t})_{ij}+h_{t}\delta_{ij})$.

In view of the fact that $U(F(\xi,t),t)$ is differentiable with respect to $t$. Since $|\nabla U(F(\xi,t),t)|=-\nabla U(F(\xi,t),t)\cdot \xi$, $F(\xi,t)=\sum_{i}(h_{t})_{i}e_{i}+h_{t}\xi$, and $\dot{F}(\xi,t)=\sum_{i}\phi_{i}e_{i}+\phi\xi$, then we get
 \begin{equation}
\begin{split}
\label{GraU}
&\frac{d}{dt}(|\nabla U(F(\xi,t),t)|)\\
&=-\nabla^{2}U(F(\xi,t),t)\xi\cdot \dot{F}(\xi,t)-\nabla \dot{U}(F(\xi,t),t)\cdot \xi\\
&=-\nabla^{2}U(F(\xi,t),t))\xi\cdot (\phi_{i}e_{i}+\phi\xi)-\nabla \dot{U}(F(\xi,t),t)\cdot \xi.
\end{split}
\end{equation}
Applying Lemma \ref{UPO} into \eqref{GraU}, we have
 \begin{equation}
\begin{split}
\label{GraU2}
&\frac{d}{dt}(|\nabla U(F(\xi,t),t)|)\\
&=-\nabla^{2}U(F(\xi,t),t)\xi\cdot (\phi_{i}e_{i}+\phi\xi)-\nabla \dot{U}(F(\xi,t),t)\cdot \xi\\
&=\kappa c^{t}_{ij}|\nabla U|_{j}\phi_{i}-\kappa |\nabla U|c^{t}_{ii}\phi-\nabla \dot{U}(F(\xi,t),t)\cdot \xi+2\phi.
\end{split}
\end{equation}
Substituting \eqref{GraU2} into \eqref{Gcal}, we have
 \begin{equation}
\begin{split}
\label{Gcal2}
&\frac{d}{dt}G(h_{t})\\
&=|\nabla U(F(\xi,t),t)|^{2}c^{t}_{ij}(\phi_{ij}+\phi\delta_{ij})\\
&\quad+2|\nabla U(F(\xi,t),t)|\det((h_{t})_{ij}+h_{t}\delta_{ij})[\kappa c^{t}_{ij}|\nabla U|_{j}\phi_{i}-\kappa |\nabla U(F(\xi,t),t)|c^{t}_{ii}\phi-\nabla \dot{U}(F(\xi,t),t)\cdot \xi+2\phi]\\
&=-2\frac{|\nabla U(F(\xi,t),t)|}{\kappa}\nabla \dot{U}(F(\xi,t),t)\cdot \xi-2|\nabla U(F(\xi,t),t)|^{2}c^{t}_{ii}\phi+4\phi\frac{|\nabla U(F(\xi,t),t)|}{\kappa}\\
&\quad +2c^{t}_{ij}|\nabla U(F(\xi,t),t)||\nabla U(F(\xi,t),t)|_{j}\phi_{i}+|\nabla U(F(\xi,t),t)|^{2}c^{t}_{ij}(\phi_{ij}+\phi\delta_{ij}).
\end{split}
\end{equation}
For simplifying \eqref{Gcal2}, we apply the conclusion that  $\Sigma_{j}c_{ijj}=0$ proved by \cite{CY76}, where $c_{ijl}$ is the covariant derivative tensor of $c_{ij}$.

One see
 \begin{equation}
\begin{split}
\label{Gcal3}
&\frac{d}{dt}G(h_{t})\\
&=-2\frac{|\nabla U(F(\xi,t),t)|}{\kappa}\nabla \dot{U}(F(\xi,t),t)\cdot \xi-|\nabla U(F(\xi,t),t)|^{2}c^{t}_{ii}\phi\\
&\quad +(c^{t}_{ij}|\nabla U(F(\xi,t),t)|^{2}\phi_{i})_{j}+4\phi\frac{|\nabla U(F(\xi,t),t)|}{\kappa}\\
&:=\mathcal{L}^{t}(\phi),
\end{split}
\end{equation}
 we hope that the operator $\mathcal{L}^{t}$ is self-adjoint on $L^{2}({\sn},d\xi)$, to realize it, for $j=1,2,3,4$, set $\mathcal{L}^{t}_{j}$ as
 \begin{equation*}\label{L1}
\mathcal{L}^{t}_{1}=-2\frac{|\nabla U(F(\xi,t),t)|}{\kappa}\nabla \dot{U}(F(\xi,t),t)\cdot \xi,
\end{equation*}
\begin{equation*}\label{L2}
\mathcal{L}^{t}_{2}=-|\nabla U(F(\xi,t),t)|^{2}c^{t}_{ii}\phi,
\end{equation*}
\begin{equation*}\label{L3}
\mathcal{L}^{t}_{3}=(c^{t}_{ij}|\nabla U(F(\xi,t),t)|^{2}\phi_{i})_{j},
\end{equation*}
and
\begin{equation*}\label{L4}
\mathcal{L}^{t}_{4}=4\phi\frac{|\nabla U(F(\xi,t),t)|}{\kappa}.
\end{equation*}
It is necessary to verify that the operators $\mathcal{L}^{t}_{1}, \mathcal{L}^{t}_{2}, \mathcal{L}^{t}_{3}$ and $\mathcal{L}^{t}_{4}$ are self-adjoint on $L^{2}({\sn},d\xi)$ respectively, i.e., for $\phi_{1}(\xi), \phi_{2}(\xi)\in \mathcal{S}$, satisfying
\begin{equation}\label{selfad}
\int_{{\sn}}\phi_{1}(\xi)\mathcal{L}^{t}_{j}(\phi_{2}){d}{\xi}=\int_{{\sn}}\phi_{2}(\xi)\mathcal{L}^{t}_{j}(\phi_{1}){d}{\xi}, \quad j=1,2,3,4.
\end{equation}
Obviously, $\mathcal{L}^{t}_{2}$, $\mathcal{L}^{t}_{4}$ is self-adjoint, and $\mathcal{L}^{t}_{3}$ is self-adjoint by applying integration by parts. Now, we focus on $\mathcal{L}^{t}_{1}$.

Since $\nabla\cdot (\nabla U_{i}(X,t,s))=-2$ in $\Omega_{i,t,s}$, and $U_{i}(\cdot)$ is differentiable with regard to $s$, which can be proved along similar lines showed by \cite[Lemma 3.1]{C15}, thus at $s=0$, we can get
\begin{equation*}\label{Uderva1}
\left\{
\begin{array}{lr}
\nabla \cdot (\nabla \dot{U}_{1})= 0,  \\
\nabla \cdot (\nabla \dot{U}_{2})= 0.
\end{array}\right.
\end{equation*}
So,
 \begin{equation}
\begin{split}
\label{Uderva2}
0&=\dot{U}_{2}\nabla \cdot(\nabla\dot{ U}_{1})-\dot{U}_{1}\nabla \cdot (\nabla \dot{U}_{2})\\
&=\nabla\cdot (\nabla \dot{U}_{1} \dot{U}_{2})-\nabla\cdot (\nabla \dot{U}_{2} \dot{U}_{1}).
\end{split}
\end{equation}
Here $\dot{U}_{i}:=\dot{U}_{i}(X,t)=\frac{\partial}{\partial t}(U_{i}(X,t))$. Employing \eqref{Uderva2}, we have
\begin{align}\label{Uderva3}
&\int_{\partial \Omega_{t}}\dot{U}_{1}(X,t)(\nu_{\Omega_{t}}(X)\cdot \nabla \dot{U}_{2}(X,t)){d}{\mathcal{H}^{n-1}(X)}\notag\\
&=\int_{\partial \Omega_{t}}\dot{U}_{2}(X,t)(\nu_{\Omega_{t}}(X)\cdot \nabla \dot{U}_{1}(X,t)){d}{\mathcal{H}^{n-1}(X)}.
\end{align}
Due to $U_{i}(F(\xi,t,s),t,s)=0$ on $\partial \Omega_{i,t,s}$ and $U_{i}(\cdot)$ is differentiable with respect to $s$, we take the differentiation  of both sides with respect to $s$ and find at $s=0$,
\begin{equation*}\label{Uderva4}
\dot{U}_{i}(F(\xi,t),t)+\nabla U(F(\xi,t),t)\cdot \dot{F}(\xi,t)=0,
\end{equation*}
it is further deduced as
\begin{align}\label{Uderva5}
\dot{U}_{i}(F(\xi,t),t)&=-\nabla U(F(\xi,t),t)\cdot (\sum_{k}(\phi_{i})_{k}e_{k}+\phi_{i}\xi)\notag\\
&=|\nabla U(F(\xi,t),t)|\xi\cdot(\sum_{k}(\phi_{i})_{k}e_{k}+\phi_{i}\xi)\notag\\
&=|\nabla U(F(\xi,t),t)|\phi_{i}.
\end{align}
Hence, for $X\in \partial \Omega_{t}$, $\dot{U_{i}}(X,t)=|\nabla U(X,t)|\phi_{i}(\nu_{\Omega_{t}}(X))$, substituting ~\eqref{Uderva5} into ~\eqref{Uderva3}, we get
 \begin{equation}
\begin{split}\label{L1selfad}
&\int_{\partial \Omega_{t}}\phi_{1}(\nu_{\Omega_{t}}(X))|\nabla U(X,t)|(\nu_{\Omega_{t}}(X)\cdot \nabla \dot{U}_{2}(X,t)){d}{\mathcal{H}^{n-1}(X)}\\
&=\int_{\partial \Omega_{t}}\phi_{2}(\nu_{\Omega_{t}}(X))|\nabla U(X,t)|(\nu_{\Omega_{t}}(X)\cdot \nabla \dot{U}_{1}(X,t)){d}{\mathcal{H}^{n-1}(X)},
\end{split}
\end{equation}
 \eqref{L1selfad} indicates that $\mathcal{L}^{t}_{1}$ is self-adjoint.

From above results, we conclude that, $\mathcal{L}^{t}$ is self-adjoint on $L^{2}({\sn},d\xi)$. In view of \eqref{Fdef1*}, $G(h_{t})$ is positively homogeneous of order $(n+1)$, if $\phi=h_{t}$, $G(h_{t}+sh_{t})=G((1+s)h_{t})=(1+s)^{n+1}G(h_{t})$, then we have
\begin{equation}\label{GVhom}
\frac{d}{ds}G((1+s)h_{t})\Big|_{s=0}=(n+1)G(h_{t})=\mathcal{L}^{t}h_{t}.
\end{equation}
Applying \eqref{GVhom} into \eqref{capcalcu}, we obtain
\begin{equation}
\begin{split}
\label{Fcapcalcu}
\frac{d}{dt}f(t)&=\frac{1}{n+2}\int_{{\sn}}(\phi G(h_{t})+h_{t}\mathcal{L}^{t}\phi){d}{\xi}\\
&=\frac{1}{n+2}\int_{{\sn}}(\phi G(h_{t})+\phi\mathcal{L}^{t}h_{t}){d}{\xi}\\
&=\frac{1}{n+2}\int_{{\sn}}(\phi G(h_{t})+(n+1)\phi G(h_{t})){d}{\xi}\\
&=\int_{{\sn}}\phi G(h_{t}){d}{\xi},
\end{split}
\end{equation}
namely,
\begin{align*}\label{Fcapcalcu2}
f^{'}(t)&=\int_{{\sn}}\phi(\xi)|\nabla  U(F(\xi,t),t)|^{2}\det((h_{t})_{ij}(\xi)+h_{t}(\xi)\delta_{ij}){d}{\xi}.
\end{align*}
Using \eqref{Fcapcalcu}, we get
\begin{equation*}
\begin{split}
\label{Up3}
f^{''}(0)&=\int_{\sn}\phi\mathcal{L} \phi d\xi\\
&=-2\int_{\sn}\phi \frac{|\nabla U(F(\xi))|}{\kappa}\nabla \dot{U}(F(\xi))\cdot \xi d \xi-\int_{\sn}\sum_{i}|\nabla U(F(\xi))|^{2}c_{ii}\phi^{2}d \xi\\
&\quad +\int_{\sn}\sum_{i,j}\phi(c_{ij}|\nabla U(F(\xi))|^{2}\phi_{i})_{j}d \xi+4\int_{\sn}\phi^{2}\frac{|\nabla U(F(\xi)|}{\kappa}d\xi.
\end{split}
\end{equation*}
Hence, Proposition \ref{fsd} is completed.

{\bf Proof of Theorem \ref{main2}}. Assume first that $\phi\in C^{\infty}(\sn)$. Let $\varepsilon>0$ be such that $h+t\phi\in \mathcal{S}$ for every $t\in (-\varepsilon,\varepsilon)$. Set  $f(t)=T(h+t\phi)$ and $\phi(t)=f(t)^{1/(n+2)}$ for $t\in (-\varepsilon,\varepsilon)$. Recall Lemma \ref{P}, we know that $\phi$ is concave, then we get
\begin{equation*}\label{}
\phi^{''}(0)=\frac{1}{n+2}\left( \frac{1}{n+2}-1 \right)f(0)^{\frac{1}{n+2}-2}f^{'}(0)^{2}+\frac{1}{n+2}f(0)^{\frac{1}{n+2}-1}f^{''}(0)\leq 0.
\end{equation*}
If $f^{'}(0)=0$, i.e., the first equality of \eqref{PE} holds, thus $f^{''}(0)\leq 0$. Then, we have
\begin{equation}
\begin{split}
\label{UR}
&-\int_{\sn}{\rm tr}(c_{ij})|\nabla U(F(\xi))|^{2}\phi^{2}d \xi-2\int_{\sn}\phi\frac{|\nabla U|}{\kappa}\nabla \dot{U}(F(\xi))\cdot \xi d \xi+4\int_{\sn}\phi^{2}\frac{|\nabla U(F(\xi))|}{\kappa}d\xi\\
&\quad \leq -\int_{\sn}\phi \sum_{i,j}(c_{ij}|\nabla U(F(\xi))|^{2}\phi_{i})_{j}d \xi.
\end{split}
\end{equation}
So we get inequality \eqref{T32} by using integration by parts in the r.h.s of \eqref{UR}. The general case $\phi\in C^{1}(\sn)$ follows by a standard approximation argument. The proof is completed.
\end{proof}
\begin{rem}\label{rme}
Denote by  $\xi_{0}$ a fixed vector in $\rnnn$. Let $h\in \mathcal{S}$ and $\phi(\xi)=\xi \cdot \xi_{0}$ for $\xi\in \sn$. Since $T(\cdot)$ is translation invariant, we have
\[
T(h+t\phi)=T(\Omega+t\xi_{0})=T(\Omega) \quad \forall t\in \R.
\]
Then we get $f^{'}(0)=0$ and $f^{''}(0)=0$. In view of Theorem \ref{main2}, we conclude that \eqref{T31} is satisfied and that \eqref{T32} is an equality.
\end{rem}
{\bf Proof of Theorem \ref{main}}. Let $\psi \in C^{1}(\partial \Omega)$ and set $\phi(\xi)=\psi(F(\xi))$ for every $\xi\in \sn$, then $\phi \in C^{1}(\sn)$. By applying $\xi=\nu_{\Omega}(X)$, we get
 \begin{equation*}\label{}
\int_{\sn}\phi(\xi)\det(h_{ij}(\xi)+h(\xi)\delta_{ij})d \xi=\int_{\partial \Omega}\psi d \mathcal{H}^{n-1}(X).
\end{equation*}
Analogously,
\begin{equation}\label{re1}
\int_{\sn}{\rm tr}(c_{ij}(\xi))|\nabla U(F(\xi))|^{2}\phi^{2}(\xi)d\xi=\int_{\partial \Omega}{\rm tr}(D \nu_{\Omega})|\nabla U(X)|^{2}\psi^{2}d \mathcal{H}^{n-1}(X),
\end{equation}
and
\begin{equation*}
\begin{split}
\label{}
\sum_{i,j}c_{ij}(\xi)|\nabla U(F(\xi))|^{2}\phi_{i}(\xi)\phi_{j}(\xi)
&=\det(h_{ij}+h\delta_{ij})(\xi)|\nabla U(F(\xi))|^{2}(\nabla_{\sn} \phi(\xi)\cdot D\nu_{\Omega}\nabla_{\sn} \phi(\xi)).
\end{split}
\end{equation*}
Then
\begin{equation}
\begin{split}
\label{re2}
\int_{\sn}\sum_{i,j}c_{ij}(\xi)|\nabla U(F(\xi))|^{2}\phi_{i}(\xi)\phi_{j}(\xi)d \xi&=\int_{\partial \Omega}|\nabla U(X)|^{2}((D\nu_{\Omega})^{-1}\nabla_{\partial \Omega} \psi\cdot \nabla_{\partial \Omega} \psi)d \mathcal{H}^{n-1}(X).
\end{split}
\end{equation}
Combining  \eqref{re1}, and \eqref{re2} with Theorem \ref{main2} and ${\rm II}_{\partial \Omega}=D\nu_{\Omega}$, we conclude that Theorem \ref{main} holds. Hence, Theorem \ref{main} is completed.


\end{document}